\documentclass[11pt]{article}
\usepackage[a4paper]{anysize}\marginsize{3.5cm}{3.5cm}{1.3cm}{2.5cm}
\usepackage[latin1]{inputenc}
\usepackage{amsmath}

\pdfpagewidth=\paperwidth \pdfpageheight=\paperheight
\usepackage{amssymb}

\makeatletter\renewcommand\@seccntformat[1]{\csname the#1\endcsname.\enspace}\makeatother
\renewenvironment{abstract}{\begin{quote}\hrulefill\par\footnotesize\textbf{\abstractname.}}{\par\vskip-0.5\baselineskip\hrulefill\end{quote}}

\newtheorem{introtheorem}{Theorem}  

\newtheorem{thm}{Theorem}[section]
\newtheorem{theorem}[thm]{Theorem}
\newtheorem{lemma}[thm]{Lemma}
\newtheorem{proposition}[thm]{Proposition}
\newtheorem{corollary}[thm]{Corollary}

\newcommand\mkthm[2]{\newenvironment{#1}{\begin{#2}\rm}{\end{#2}}}
 \mkthm{definition}{tdefinition}
\newtheorem{remark}[thm]{Remark}
       \mkthm{example}{texample}
     \mkthm{question}{tquestion}

\newtheorem{thevarthm}[thm]{\varthmname}

\newenvironment{varthm*}[1]{\trivlist\item[]{\bf #1.}\it}{\endtrivlist}

\newenvironment{proof}[1][Proof]{\trivlist\item[\hskip\labelsep{\textit{#1.}}]}{\hspace*{\fill}$\Box$\endtrivlist}
\newenvironment{proofT}[1][Proof of the theorem]{\trivlist\item[\hskip\labelsep{\textit{#1.}}]}{\hspace*{\fill}$\Box$\endtrivlist}

\let\hat=\widehat

\renewcommand\O{\mathcal O}
\renewcommand\bar{\overline}

\newcommand\smallmatr[2][*{20}{c}]{{\fontsize{8}{9}\arraycolsep=4pt\selectfont\left(\!\!\begin{array}{#1}#2\end{array}\!\!\right)}}

\renewcommand\ge{\geqslant}  
\renewcommand\le{\leqslant}  

\newcommand\keywords[1]{{\renewcommand\thefootnote{}\footnotetext{\textbf{Keywords:} #1.}}}
\newcommand\subclass[1]{{\renewcommand\thefootnote{}\footnotetext{\textbf{Mathematics Subject Classification (2010):} #1.}}}

\newcommand\C{\mathbb C}
\newcommand\Q{\mathbb Q}
\newcommand\R{\mathbb R}
\newcommand\Z{\mathbb Z}
\newcommand\N{\mathbb N}

\newcommand\be[1][@{\;}r@{\;}c@{\;}l@{\;}l@{\;}]{$$\everymath{\displaystyle}\renewcommand\arraystretch{1.2}\begin{array}{#1}}
\newcommand\ee{\end{array}$$}

\newcommand\engqq[1]{``#1''}
\newcommand\compact{\itemsep=0cm \parskip=0cm}
\newcommand\set[1]{\left\{#1\right\}}

\newcommand\with{\,\,\vrule\,\,}
\newcommand\matr[1]{\left(\begin{array}{*{20}{c}} #1 \end{array}\right)}

\newcommand\tmod[1]{\;({\rm mod}~#1)}

\newcommand\abs[1]{\left|#1\right|}

\newenvironment{bycases}{\left\{\begin{array}{@{}l@{\quad}l}}{\end{array}\right.}

\newcommand\inverse{^{\smash-\mkern-1mu1}}
\newcommand\isom{\simeq}  
\newcommand\transpose{^t}
\newcommand\divides{\mid}

\newcommand\newop[2]{\newcommand#1{\mathop{\rm #2}\nolimits}}

\newop\Diag{Diag}
\newop\Hom{Hom}
\newop\GL{GL}
\newop\LCM{LCM}
\newop\Vol{Vol}
\def\VolNef(#1,#2){\Vol(\Nef(#1), #2)}
\newop\Pos{Pos}
\newop\Amp{Amp}
\newop\Bigcone{Big}
\newop\End{End}
\newcommand\Endsym{\End^{\rm sym}}
\newcommand\EndQ{\End_\Q}
\newcommand\EndQsym{\EndQ^{\rm sym}}
\newop\Nef{Nef}
\newop\NS{NS}
\newop\NSR{NS_{\R}}
\newcommand\rhoan{\rho_{\rm an}}  

\begin{document}

   \title{Nef cone volumes and discriminants of abelian surfaces}
   \author{\normalsize Thomas Bauer, Carsten Borntr\"ager}
   \date{\normalsize Version of August 25, 2016}
   \maketitle
   \thispagestyle{empty}
   \subclass{14C20, 14K05}

\begin{abstract}
   The nef cone volume appeared first in work of Peyre
   in a number-theoretic context on Del Pezzo
   surfaces, and it was studied by Derenthal and co-authors
   in a series of papers.
   The idea
   was subsequently extended
   to also measure the Zariski chambers of Del Pezzo
   surfaces.
   We start in this paper to explore the possibility
   to use this attractive concept to effectively measure
   the size of the nef cone on
   algebraic surfaces in general.
   This provides an
   interesting way of measuring
   in how big a space an ample line bundle can be moved without
   destroying its positivity.
   We give here complete results for simple
   abelian surfaces that admit a principal polarization and for
   products of elliptic curves.
\end{abstract}


\section*{Introduction}

   The concept of nef cone volume appears first in work of
   Peyre~\cite{Peyre:hauteurs} in
   a number-theoretic context (Manin's conjecture)
   on Del Pezzo surfaces,
   and was subsequently studied by Derenthal and co-authors in a series
   of papers
   \cite{Derenthal:on-a-constant,Derenthal:quintic,Derenthal:nef-cone-volume,
         Derenthal-Browning:cubic,Derenthal-Browning:quartic,Derenthal-Joyce-Teitler:nef}.
   It was shown in \cite{BS:volumes} that an extension of this
   notion
   can be used,
   beyond
   its
   number-theoretical origins,
   to measure the size of
   Zariski chambers (and also of the whole big cone) on Del Pezzo
   surfaces.
   In the present paper we start exploring the possibility
   to use this attractive concept to effectively measure
   the size of the nef cone (or the big cone) on
   algebraic surfaces in general.

   To begin with,
   it is natural to generalize the construction
   from \cite{Derenthal:nef-cone-volume}
   by
   fixing an
   ample line bundle $H$ on $X$ (which was $-K_X$ in the Del
   Pezzo case) and proceeding in the following way:
   Using the half-space
   \be
      H^{\le 1} = \set{L\in\NSR(X)\with H\cdot L\le 1}
   \ee
   we define
   for any cone $\mathcal C\subset\NSR(X)$
   its \emph{cone volume with respect to $H$}
   as the volume of the truncated cone
   $\mathcal C \cap H^{\le 1}$. We will use the notation
   \be
      \Vol(\mathcal C,H):=\Vol(\mathcal C \cap H^{\le 1})
      \,.
   \ee
   Here
   the volume
   on the right-hand side
   is taken after
   identifying the N\'eron-Severi vector space $\NSR(X)$
   with $\R^{\rho(X)}$
   by means of an isomorphism induced by a
   lattice basis. (In other words, we normalize the volume
   on $\NSR(X)$
   by requiring that
   a
   fundamental parallelotope of the lattice $\NS(X)$
   has volume one).
   The quantity is then independent of the choice of
   the lattice basis.
   In particular,
   the \emph{nef cone volume of $X$ with respect
   to $H$},
   \be
      \VolNef(X, H)
      \,,
   \ee
   is then an invariant that is naturally attached to the
   polarized surface $(X,H)$.
   Geometrically, we may think of it as telling us
   in how big a space an ample line bundle can be moved without
   destroying its positivity.

   For simple principally polarized
   abelian surfaces and for products of two elliptic
   curves, we develop in this paper a complete picture of how
   nef cone
   volumes are behaved.

\begin{introtheorem}\label{thm:intro-simple}
   Let $X$ be a simple abelian surface which admits a principal
   polarization, and let $H$ be any ample line bundle on $X$ (not
   necessarily a multiple of the principal polarization).
   Then the nef cone volume of $X$ with respect to $H$ can be
   determined in terms of $\End(X)$, specifically:

   \begin{itemize}
   \item[\rm(a)]
      Suppose that $X$ has only integer multiplication, i.e.,
      $\End(X)=\Z$. Then
      \be
         \VolNef(X,H)=\frac{1}{\sqrt{2H^2}}.
      \ee
   \item[\rm(b)]
      Suppose that $X$ has real
      multiplication, i.e., $\EndQ(X)=\Q(\sqrt d)$ for some square-free integer
      $d>0$, and let $f\ge 1$ be the conductor of $\End(X)$ in
      $\EndQ(X)$ (see Sect.~\ref{sect:simple} for
      details). Then
      \be
         \VolNef(X,H)=
         \begin{bycases}
            \frac1{2f\sqrt d (H^2)} & \mbox{ if } d\equiv 2,3 \tmod 4 \\[\bigskipamount]
            \frac1{f\sqrt d (H^2)} & \mbox{ if }  d\equiv 1 \tmod 4
         \end{bycases}
      \ee
      The same formula applies when $X$ has complex
      multiplication: In that case the numbers $f$ and $d$
      are to be taken from the order
      $\Endsym(X)$ in the real-quadratic subfield $\EndQsym(X)=\Q(\sqrt d)\subset\EndQ(X)$.
   \item[\rm(c)]
      Suppose that $X$ has indefinite quaternion multiplication, and write
      $\End_\Q(X)=\Q+i\Q+j\Q+ij\Q$,
      and
      $\End(X)=\Z\oplus \Z a \oplus \Z b \oplus \Z ab$
      with suitable primitive Rosati invariant elements.
      Then
      \be
         \VolNef(X,H)=\frac{\pi\sqrt{2}}{3\sqrt{|\det S_\delta(a,b)|}(H^2)^\frac{3}{2}}.
      \ee
      (see Sect.~\ref{sect:simple} for details, in particular
      for the definition of the matrix $S_\delta(a,b)$.)
   \end{itemize}
\end{introtheorem}

\begin{introtheorem}\label{thm:intro-product}
   Let $X=E_1\times E_2$ be a product of two elliptic curves,
   and let $H$ be any ample line bundle on $X$.

   \begin{itemize}
   \item[\rm(a)]
      If $E_1$ and $E_2$ are not isogenous, then
      \be
         \VolNef(X,H)=\frac{1}{(H^2)}
         \,.
      \ee
   \item[\rm(b)]
      Suppose that $E_1$ and $E_2$ are isogenous and have no complex
      multiplication, i.e. $\End(E_1)=\End(E_2)=\Z$.
      If $\sigma:E_1\to E_2$ is an isogeny of minimal degree,
      then
      \be
         \VolNef(X,H)=\frac{\pi}{3\cdot\sqrt{2\cdot \deg(\sigma)}\cdot (H^2)^{\frac{3}{2}}}
         \,.
      \ee

   \item[\rm(c)]
      Suppose that
      $E_1$ and $E_2$ are isogenous and have complex
      multiplication.
      Write $\EndQ(E_i)=\Q(\sqrt d)$ with a square-free integer
      $d<0$, and let $f_1$ and $f_2$ the conductors of
      $\End(E_1)$ and $\End(E_2)$, respectively.
      Then
      \be
         \VolNef(X,H)=
         \begin{bycases}
           \frac{\pi}{6\cdot \LCM(f_1,f_2) \sqrt{|d|}(H^2)^2} & \mbox{ if } d\equiv 2,3 \tmod 4 \\[\bigskipamount]
           \frac{\pi}{3\cdot \LCM(f_1,f_2) \sqrt{|d|}(H^2)^2} & \mbox{ if }  d\equiv 1 \tmod 4.
         \end{bycases}
      \ee
   \end{itemize}

\end{introtheorem}

\bigskip
\noindent
   Our work is based on the observation
   that the volume of the positive cone
   \be
      \Pos(X)=\set{D\in \NSR(X)\with D^2\ge 0,\ H\cdot D\ge 0}
   \ee
   is governed (on any smooth projective surface)
   by the discriminant of the N\'eron-Severi lattice
   (see Prop~\ref{prop:volume-pos-cone}).
   In the case of
   abelian surfaces, the positive cone and the
   pseudo-effective cone both coincide with the nef cone.
   Therefore,
   the issue in
   determining
   the nef cone volume
   is to find effective ways to
   determine the
   discriminant.
   We show for simple abelian surfaces how the discriminant depends on
   the structure of the endomorphism ring
   and we determine it in terms of the ring-theoretic data
   (see
   Section~\ref{sect:simple}).
   For products of elliptic curves $E_1\times E_2$, the main case
   is when $E_1$ and $E_2$ are isogenous curves with complex
   multiplication; the crucial point here is to express the
   discriminant in terms of the conductors
   (see Theorem~\ref{thm:dicriminant-conductors}).

   Note that, as we make essential use of the isomorphism
   of $\NS(X)$ with the group $\Endsym(X)$ of symmetric
   endomorphisms, our present methods
   shed no light on the non-principally polarized case.
   The question of nef cone volumes is of course
   equally of interest there, and it would be very
   interesting to see how the picture extends.


\section{The volume of the positive cone on algebraic surfaces}

   The
   \textit{positive cone} of
   a smooth projective surface
   $X$ is by definition
   \be
      \Pos(X)=\set{D\in \NSR(X)\with D^2\ge 0,\ H\cdot D\ge 0}
   \ee
   where $H$ is a fixed ample divisor.
   We show here that its cone volume can be computed in terms
   of the self-intersection of $H$ and the discriminant of the
   N\'eron-Severi lattice:

\begin{proposition}\label{prop:volume-pos-cone}
   Let $X$ be a smooth projective surface and $H$ an ample divisor.
   Denote by $\rho$ the Picard number of $X$ and by
   $\Delta$ the discriminant of the N\'eron-Severi group
   (i.e., the determinant of the Gram matrix with respect to
   a lattice basis of $\NS(X)$).
   Then
   the cone volume
   of the positive cone is given by the formula
   \be
      \Vol(\Pos(X), H)
      = \frac{V_{\rho}}{\sqrt{\abs\Delta}\cdot (H^2)^{\rho/2}}
   \ee
   where $V_{\rho}$ is the volume of the truncated cone
   \be
      \set{x\in\R^\rho\with 0\le x_1\le 1 \mbox{ and } x_1^2-x_2^2-\dots-x_\rho^2\ge 0}
      \,.
   \ee
\end{proposition}

\begin{remark}
   For abelian surfaces we have $1\le\rho\le 4$, so in that case
   we will only need
   the following numerical values of $V_\rho$:
   \be
      V_1=V_2=1, \quad
      V_3=V_4=\frac\pi3
   \ee
\end{remark}

\begin{proof}[Proof of the proposition]
   Let $S$ be the matrix of the intersection form with respect to
   a lattice basis.
   We can find an $\R$-basis $B_1,\dots,B_\rho$
   of $\NSR(X)$ with respect to which the intersection form
   is diagonal
   and where
   $B_1=1/\sqrt{H^2}\cdot H$.
   Let $T\in\GL(n,\R)$ be a change of base matrix satisfying
   $T\transpose ST=\Diag(1,-1,\dots,-1)$. In particular, we have then
   $\abs{\det T}=1/\sqrt{\abs\Delta}$, as $\Delta=\det S$.
   The truncated positive cone
   \be
      \Pos(X)\cap H^{\le 1}
   \ee
   (more precisely, its coordinate set with respect to the lattice basis)
   is mapped by $T\inverse$ to the set
   \be
      \set{x\in\R^\rho\with 0\le x_1\le \tfrac1{\sqrt{H^2}} \mbox{ and } x_1^2-x_2^2-\dots-x_\rho^2\ge 0}
   \ee
   Upon multiplying this set by $\sqrt{H^2}$, we obtain the
   subset of $\R^\rho$ given
   in the statement of the proposition. Therefore
   the volume of the latter set is given by
   \be
      V_\rho=
      (\sqrt{H^2})^\rho
      \cdot
      \abs{\det{T\inverse}}
      \cdot
      \Vol(\Pos(X)\cap H^{\le 1})
   \ee
   and this implies the assertion.
\end{proof}

   Thanks to the inclusions of cones
   \be
      \Nef(X) \subset \Pos(X) \subset \bar{\Bigcone}(X)
   \ee
   the proposition immediately yields the following estimates:

\begin{corollary}
   In the situation of Prop~\ref{prop:volume-pos-cone}
   we have
   \be
      \Vol(\Nef(X), H)
      \le \frac{V_{\rho}}{\sqrt{\abs\Delta}\cdot (H^2)^{\rho/2}}
      \le \Vol(\bar{\Bigcone}(X), H)
      \,.
   \ee
\end{corollary}
   This shows in particular that the nef cone volume can become
   arbitrarily small since $V_\rho/\sqrt{|\Delta|}$ can be arbitrarily small.
   (This happens for suitable
   abelian surfaces with real multiplication, see Sect.~\ref{sect:simple}.)

   On abelian surfaces, the positive cone and the
   pseudo-effective cone both coincide with the nef cone.
   Therefore, as a consequence of
   Prop.~\ref{prop:volume-pos-cone},
   determining nef cone volumes of abelian surfaces
   is equivalent to determining their discriminants:

\begin{corollary}
   Let $(X, H)$ be a polarized abelian surface. Then the nef cone volume of $X$
   is given by
   \be
      \VolNef(X,H)
      = \frac{V_{\rho}}{\sqrt{\abs\Delta}\cdot (H^2)^{\rho(X)/2}}
      \,.
   \ee
\end{corollary}


\section{Volumes of simple abelian surfaces}\label{sect:simple}

   Let in this section
   $X$ be a simple abelian surface admitting a principal
   polarization~$L_0$.
   Our purpose it to determine the nef cone volume
   \be
      \VolNef(X, H)
   \ee
   with respect to any ample line bundle $H$ on $X$.
   We will make use of the classification of endomorphism algebras
   of abelian varieties (see~\cite{BL:CAV}), and work
   according to whether $X$ has integer, real, complex or
   quaternion multiplication.
   The results in this section prove
   Theorem~\ref{thm:intro-simple} from the introduction.

   \paragraph{Type 0: Integer multiplication.}
   Suppose that $\End(X)=\Z$.
   As $L_0$ gives a basis of $\NS(X)$,
   we have
   $\Delta(X)=2$ for the discriminant.
   Prop~\ref{prop:volume-pos-cone} then implies
   \be
      \VolNef(X,H)=\frac{1}{\sqrt{2}\cdot\sqrt{H^2}}
      \,.
   \ee
   Alternatively,
   the result can in this simple case
   be obtained more directly
   by noting that
   for $H=c\cdot L_0$ with $c>0$ we have:
   \be
      H^{\le 1}=\set{L\in\Nef(X)\with H\cdot L\le 1}
      \isom\set{x\in\R\with x\le\tfrac{1}{2c}}
   \ee
   and hence the truncated nef cone is
   \be
      \Nef(X)\cap H^{\le 1}\isom [0,\tfrac{1}{2c}]
   \ee
   which has volume
   $\frac{1}{2c}=1/(\sqrt{2}\cdot\sqrt{H^2})$.

\paragraph{Type 1: Real multiplication.}
   Suppose that $X$ has real multiplication, i.e., that
   $\End_\Q(X)=\Q(\sqrt d)$ for some square-free integer $d>0$.
   The endomorphism ring is an order in $\End_\Q(X)$, and hence
   of the form $\End(X)=\Z+f\omega\Z$, where $f\ge 1$ is an
   integer and
   \be
      \omega=
      \begin{bycases}
         \sqrt d            & \mbox{ if } d\equiv 2,3 \tmod 4 \\
         \tfrac12(1+\sqrt d) & \mbox{ if } d\equiv 1 \tmod 4
      \end{bycases}
   \ee
   The isomorphism of groups
   \be
      \varphi:\NS(X)\to\End(X), \quad L\mapsto
      \phi_{L_0}\inverse\phi_L
   \ee
   provides us with a lattice basis of $\NS(X)$, given by
   $L_0=\varphi\inverse(1)$ and $L_{f\omega}:=\varphi\inverse(f\omega)$.
   The intersection matrix of this basis is
   \be
      \matr{2 & 0 \\
      0       & -2f^2d}
      \quad\mbox{if } d\equiv 2,3 \tmod 4
   \ee
   and
   \be
      \matr{2 & f \\
      f       & \tfrac12 f^2(1-d)}
      \quad\mbox{if } d\equiv 1 \tmod 4
   \ee
   This follows by considering the characteristic polynomial
   of $f\omega$ in $\Q(\sqrt d)$
   (which coincides with the analytic characteristic polynomial
   of the endomorphism)
   and applying
   \cite[Prop.~5.2.3]{BL:CAV}.
   Therefore the discriminant is given by
   \be
      \Delta(X)=-4f^2d \quad\mbox{and}\quad
      \Delta(X)=-f^2d
      \,,
   \ee
   respectively.
   So by Prop.~\ref{prop:volume-pos-cone} we obtain
   for any ample line bundle $H$
   \be
      \VolNef(X,H)=
      \begin{bycases}
         \frac1{2f\sqrt d (H^2)} & \mbox{ if } d\equiv 2,3 \tmod 4 \\[\bigskipamount]
         \frac1{f\sqrt d (H^2)} & \mbox{ if }  d\equiv 1 \tmod 4
      \end{bycases}
   \ee

   The formula shows in particular that
   if we increase the \engqq{size} of the real multiplication $\sqrt d$,
   then the nef cone gets smaller at the rate of $1/\sqrt d$.
   If we increase $f$ (thus making $\End(X)$ smaller), then
   the volume decreases at the rate of $1/f$.
   This shows that nef cone volume can become arbitrarily small.

\paragraph{Type 2: Complex multiplication.}

   Suppose that $X$ has complex multiplication, i.e.,
   that $\EndQ(X)$ is isomorphic to an imaginary quadratic extension of a real quadratic number field.

   We have $\EndQsym(X)=\Q(\sqrt{d})$, where $d$ is a positive square-free integer. Also $\Endsym(X)$ has to be an order in $\EndQsym(X)$,
   and hence the arguments given for Type~1 apply here as well,
   when applied to the real-quadratic subfield $\EndQsym(X)$.

\paragraph{Type 3: Indefinite quaternion multiplication.}

   Suppose now
   that $X$ has indefinite quaternion multiplication, i.e.,
   there are $\alpha,\beta\in\Z\setminus\set 0$
   with $\alpha\ge\beta$ and $\alpha>0$ such that
   $\End_\Q(X)=\Q+i\Q+j\Q+ij\Q$, where $i$ and $j$ satisfy the relations
   $i^2=\alpha$,
   $j^2=\beta$
   and $ij=-ji$.

   By \cite[Theorem~7]{Runge} we have $\End(X)=\Z\oplus \Z a \oplus \Z b \oplus \Z ab$,
   for some primitive Rosati invariant elements $a$ and $b$.
   In a rational quaternion algebra any element $x$ satisfies the equation $x^2-t(x)x+n(x)=0$,
   where $t(x)$ and $n(x)$ are the reduced trace and norm. Since the analytic representation
   $\rhoan:\EndQ(X)\to M_2(\C)$ is a
   ring homomorphism, the matrix $\rhoan(a)$ is a root of
   $x^2-t(a)x+n(a)$ and therefore the characteristic polynomial
   of $\rhoan(a)$ coincides with $x^2-t(a)x+n(a)$.
   The same
   applies to $b$. Via the isomorphism
   \be
      \varphi:\NS(X)\to\Endsym(X)= \Z\oplus \Z a \oplus \Z b, \quad L\mapsto
      \phi_{L_0}\inverse\phi_L
   \ee
   we get a basis
   $$
      L_0=\varphi^{-1}(1),~ L_a=\varphi^{-1}(a),~ L_b=\varphi^{-1}(b)
   $$
   of $\NS(X)$ and by  \cite[Prop.~5.2.3]{BL:CAV}
   we get their intersection matrix
   \be
         \matr{2 & t(a) & t(b) \\
               t(a) & 2n(a) & n(a+b)-n(a)-n(b)\\
               t(b) & n(a+b)-n(a)-n(b) & 2n(b)
               }
   \ee
   The determinant of the matrix above turns out to be the half of the determinant of the discriminant matrix
   \be
            S_\delta(a,b)=\matr{\delta(a)& \delta(a,b)\\
                         \delta(a,b) & \delta(b)
                     },
   \ee
   where the (quaternion) discriminant form is given by
   $\delta(x,y)=t(x)t(y)-2n(x,y)$ with
   $n(x,y)=n(x+y)-n(x)-n(y)$.
   So we have
   \be
      \Delta(X)=\frac12 \det(S_\delta(a,b))
   \ee
   and therefore, we get
   by Proposition
   \ref{prop:volume-pos-cone} for any ample line bundle
   $H$:
   \be
      \VolNef(X,H)=\frac{\pi\sqrt{2}}{3\sqrt{|\det S_\delta(a,b)|}(H^2)^\frac{3}{2}}.
   \ee

\begin{remark}[Non-principally polarized case]\rm
   It would be very interesting to see how the picture extends to
   the non-principally polarized case. A natural idea is to try using
   the fact that every polarization is the pullback of a principal
   polarization: Given any polarized abelian surface
   $(X,L)$, there is an isogeny
   \be
      f: X \to X_0
   \ee
   where $X_0$ is an abelian surface carrying a principal
   polarization $L_0$ with $L=f^*L_0$.
   One might then hope to relate the discriminants
   $\Delta(X)$ and $\Delta(X_0)$ via the isogeny (e.g.~its
   degree).
   However, it seems that there is no simple relation of this kind
   --
   to illustrate this, consider the following two situations:
   \begin{itemize}
   \item
      Suppose that $L$ generates $\NS(X)$
      and is of type $(1,d)$.
      (This case is in reality uninteresting from
      a volume perspective
      because of $\rho(X)=1$.)
      So $f$ is of degree $d$. In this case,
      \be
         \Delta(X)=d\cdot\Delta(X_0)
         \,.
      \ee
   \item
      Suppose that $X$ is a product $E_1\times E_2$ of non-isogenous
      elliptic
      curves. There are isogenies
      $E_1\times E_2\to E_1\times E_2$ of arbitrarily high
      degree, but
      the discriminants
      $\Delta(E_1\times E_2)$ and $\Delta(E_1'\times E_2')$
      are both $2$, independently of the chosen isogeny.
   \end{itemize}

\end{remark}


\section{Volumes of products of elliptic curves}\label{sect:products}

   In this section we consider products $X=E_1\times E_2$ of elliptic curves.
   We will determine the nef cone volume
   $\VolNef(X, H)$
   with respect to any ample line bundle $H$ on $X$.
   We will work according to whether
   $E_1$ and $E_2$ are isogenous, and whether they have complex
   multiplication -- this amounts to three types A,B,C below.
   The results in this section prove
   Theorem~\ref{thm:intro-product}
   from the introduction.

   To begin with, recall that we
   have
   \be
      \End(X) =\left(
      \begin{array}{cc}
         \End(E_{1})       & \Hom(E_{2},E_{1}) \\
         \Hom(E_{1},E_{2}) & \End(E_{2})
      \end{array} \right).
   \ee
   We will use the following description of the symmetric
   endomorphisms:

\begin{lemma}\label{lemma:endsym}
   Let $L_0$ be the principal polarization on $X$ induced by the principal polarizations on $E_i$.
   With respect to the
   Rosati-Involution associated with $L_0$, one has
   $$
      \Endsym(X)=\left\lbrace
      \left.\left(  \begin{array}{cc}
      a & \widehat\sigma \\
      \sigma & b
      \end{array} \right)
      \right|
      a,b\in \mathbb{Z}\textnormal{ and } \sigma\in\textnormal{Hom}(E_1,E_2)
      \right\rbrace
   $$
   (where $\hat\sigma$ denotes the dual homomorphism of
   $\sigma$).
\end{lemma}

\begin{proof}
   Let $\Lambda_1,\Lambda_2\subset \C$ be the lattices defining
   $E_1,E_2$. The hermitian form belonging
   to the principal polarization on $E_i$ is given by
   $H_i(x,y)=\frac{1}{a(\Lambda_i)}x\overline{y}$, where
   $a(\Lambda_i)$ denotes the area of a period parallelogram of
   the lattice $\Lambda_i$.
   For an endomorphism
   $\alpha=\smallmatr{
    \alpha_{11} & \alpha_{12} \\
    \alpha_{21} & \alpha_{22}
   }\in\End(X)$,
   the Rosati dual $\alpha'$ is given by
   $$
      \alpha'= \phi_{L_0}^{-1}\widehat{\alpha}\phi_{L_0}
      =\left(\begin{array}{cc}
       \overline\alpha_{11} & \frac{a(\Lambda_1)}{a(\Lambda_2)}\cdot\overline\alpha_{21} \\
       \frac{a(\Lambda_2)}{a(\Lambda_1)}\cdot\overline\alpha_{12} & \overline\alpha_{22}
      \end{array}  \right)
      =\left(\begin{array}{cc}
       \overline\alpha_{11} & \widehat\alpha_{21} \\
       \widehat\alpha_{12} & \overline\alpha_{22}
      \end{array}  \right)
      \,.
   $$
   The assertion then follows, since either
   $\textnormal{End}(E_i)=\mathbb{Z}$ or
   $\textnormal{End}(E_i)=\mathbb{Z}+\mathbb{Z}\omega_i$ for some
   $\omega_i\in\C\setminus \R$.
\end{proof}


\paragraph{Type A: Products of non-isogenous elliptic curves.}
   Suppose that
   $E_1$ and $E_2$ are not
   isogenous.
   We fix the
   principal polarization
   $L_0=\O_X(F_1+F_2)$, where $F_1$ and $F_2$ are the fibers of the
   projections.
   The fibers $F_1$ and
   $F_2$ are a basis of the lattice $\NS(X)$, with intersection matrix
   $\smallmatr{0 & 1 \\ 1 & 0}$. So we have
   \be
      \Delta(X)=-1
   \ee
   and therefore, with Proposition \ref{prop:volume-pos-cone},
   we obtain for any ample $H$,
   \be
      \VolNef(X,H)=\frac{1}{(H^2)}
      \,.
   \ee


\paragraph{Type B: Products of isogenous elliptic curves without complex multiplication.}
   Let $E_1,E_2$ be isogenous elliptic curves without complex multiplication.
   Let $\sigma:E_1\to E_2$ be an isogeny of minimal degree.
   We start by observing:

\begin{lemma}
   We have
   $$
      \Hom(E_1,E_2)=\Z\cdot\sigma.
   $$
\end{lemma}

\begin{proof}
   Fix any
   isogeny $\alpha: E_2 \to E_1$. The mapping
   $\Hom(E_1,E_2)\to\End(E_1)$, $\delta\mapsto\alpha\delta$,
   is an injective homomorphism of groups. Therefore $\Hom(E_1,E_2)$ is of rank~1.
   The assertion
   then follows from the minimality of $\sigma$.
\end{proof}

   From Lemma~\ref{lemma:endsym}
   we see that
   \be
      \Endsym(X)=\left\lbrace
      \left.  \matr{
           a & b\hat{\sigma} \\
           b\sigma & c }
      \right| a,b,c \in \Z
      \right\rbrace
   \ee
   As a consequence we get, upon using the isomorphism
   \be
      \varphi:\NS(X)\to\Endsym(X), \quad L\mapsto
      \phi_{L_0}\inverse\phi_L
      \,,
   \ee
   a basis $F_1,F_2,F_3$ of $\NS(X)$ with intersection matrix
   \be
     \matr{ 0 & 1 & 0   \\
              1 & 0 & 0   \\
           0 & 0 & -2\cdot \deg(\sigma) \\

             }.
   \ee
   So
   \be
      \Delta(X)=2\cdot\deg(\sigma)
   \ee
   and hence by Prop.~\ref{prop:volume-pos-cone} we get for any ample $H$,
   \be
      \VolNef(X,H)=\frac{\pi}{3\cdot\sqrt{2\cdot \deg(\sigma)}\cdot (H^2)^{\frac{3}{2}}}
      \,.
   \ee

\begin{example}
   In the preceding formula
   every positive integer occurs
   as a minimal degree. In fact, consider for instance
   the elliptic curves $E_1$ and $E_2$ with $\Lambda_1=\Z+\Z\pi i$
   and $\Lambda_2=\Z+\Z k\pi i$ for fixed $k\in \N$.
   Multiplication by $k$ gives an isogeny $E_1\to E_2$ of degree $k$,
   and there are no isogenies $E_1\to E_2$ of lower degree.
\end{example}

\begin{remark}\rm
   A self-product $E\times E$ of an elliptic curve without complex multiplication
   is of course a special case of the preceding considerations.
   However, this case can be dealt
   directly, without referring to the endomorphism ring:
   The classes of $F_1$, $F_2$ and the diagonal
   form a lattice basis of $\NS(X)$, and their intersection matrix is
   \be
      \matr{0 & 1 & 1 \\
            1 & 0 & 1 \\
            1 & 1 & 0
            }
   \ee
   So we have $\Delta(X)=2$, and hence
   for any ample $H$,
   \be
      \VolNef(X,H)=\frac{\pi}{3\cdot\sqrt{2}\cdot (H^2)^{\frac{3}{2}}}.
   \ee
\end{remark}


\paragraph{Type C: Products of isogenous elliptic curves with complex multiplication.}
   Suppose now that
   $E_1,E_2$ are isogenous and have complex multiplication,
   i.e., $\End_\Q(E_i)=\Q(\sqrt d)$ for some square-free integer $d<0$.
   We have $\End(E_i)=\Z+f_i\omega\Z$, where $f_1,f_2\ge 1$ are
   integers and
   \be
      \omega=
      \begin{bycases}
         \sqrt d            & \mbox{ if } d\equiv 2,3 \tmod 4 \\
         \tfrac12(1+\sqrt d) & \mbox{ if } d\equiv 1 \tmod 4
      \end{bycases}.
   \ee
   The crucial result in this part is that
   one can determine the discriminant of the product
   $X=E_1\times E_2$ in terms of $\omega,f_1,f_2$, without
   explicitly referring to an isogeny:

\begin{theorem}\label{thm:dicriminant-conductors}
   The discriminant of $\NS (X)$ is
   $$
      \Delta(X) = - 4 \cdot \LCM(f_1,f_2)^2\cdot \textnormal{Im}(\omega)^2.
   $$
\end{theorem}

   With Prop.~\ref{prop:volume-pos-cone} we immediately conclude:

\begin{corollary}
   We have for any ample $H$:
   \be
      \VolNef(X,H)=
      \begin{bycases}
        \frac{\pi}{6\cdot \LCM(f_1,f_2) \sqrt{|d|}(H^2)^2} & \mbox{ if } d\equiv 2,3 \tmod 4 \\[\bigskipamount]
        \frac{\pi}{3\cdot \LCM(f_1,f_2) \sqrt{|d|}(H^2)^2} & \mbox{ if }  d\equiv 1 \tmod 4.
      \end{bycases}

   \ee
\end{corollary}

   In the remainder of this section we will prove
   Theorem~\ref{thm:dicriminant-conductors}.
   We start by noting:

\begin{lemma}
   The homomorphism  group $Hom(E_{1},E_{2})$ is a lattice in $\C$.
\end{lemma}

\begin{proof}
   For any nonzero $\lambda\in \Lambda_1$,
   the mapping $\psi:\Hom(E_{1},E_{2}) \to \Lambda_2$,
   $\sigma \mapsto \sigma \lambda$,
   is an injective homomorphism of groups. Therefore the rank of
   $\Hom(E_{1},E_{2})$ is at most $2$, so it is in any
   event of the form
   $\Hom(E_{1},E_{2})=\Z\sigma_1+\Z\sigma_2$.
   We show that $\sigma_1$ and $\sigma_2$ are linear independent over
   $\R$: Assume to the contrary that
   there is an $r\in\R\setminus\left\lbrace 0\right\rbrace $
   with $\sigma_1=r\sigma_2$.
   Then we find an $r'\in\R\setminus\left\lbrace 0\right\rbrace $
   with $f_2\omega\sigma_2=r'\sigma_2$,
   since $f_2\omega \sigma_2$ is a homomorphism $E_1\to E_2$.
   We conclude $f_2\omega=r'$, which is a contradiction since $f_2\omega\in \C\setminus \R$.
\end{proof}

   We fix now some $\R$-linear independent $\sigma_1, \sigma_1 \in \C$
   with
   $\Hom(E_{1},E_{2})=\Z\sigma_1+\Z\sigma_2$.
   Then we have
   $$
      \Hom(E_{2},E_{1})=\Z\hat{\sigma}_1+\Z\hat{\sigma}_2.
   $$
   with the dual isogenies $\hat\sigma_1$ and $\hat\sigma_2$.
   Via the isomorphism
   \be
      \varphi:\NS(X)\to\Endsym(X), \quad L\mapsto
      \phi_{L_0}\inverse\phi_L
   \ee
   and
   \be
        \Endsym(X)=\left\lbrace
         \left.  \matr{
        a & b\hat{\sigma}_1+c\hat{\sigma}_2 \\
        b\sigma_1+c\sigma_2 & d }
      \right| a,b,c,d \in \Z
      \right\rbrace
    \ee
    we get a basis $F_1,...,F_4$ of $\NS(X)$ with intersection matrix

    \be
          \matr{ 0 & 1 & 0 & 0  \\
                   1 & 0 & 0 & 0  \\
                0 & 0 & -2\sigma_1\hat{\sigma}_1 &  -(\sigma_1\hat{\sigma}_2+\sigma_2\hat{\sigma}_1)\\
                   0 & 0 & -(\sigma_1\hat{\sigma}_2+\sigma_2\hat{\sigma}_1) & -2\sigma_2\hat{\sigma}_2
                  }.
    \ee
   The determinant of the matrix above is
   $$
      \Delta=(\sigma_1\widehat{\sigma}_2+\sigma_2\widehat{\sigma}_1)^2-4\cdot\sigma_1\widehat{\sigma}_1\sigma_2\widehat{\sigma}_2=-4\cdot\textnormal{Im}(\sigma_1\widehat{\sigma}_2)^2.
   $$

\begin{lemma}\label{lemma:isom}
   Let $E'_1$ be an elliptic curve isomorphic to $E_1$ and $E'_2$ an elliptic curve isomorphic to $E_2$.
   If $\{\delta_1,\delta_2\}$ is a basis of $\Hom(E'_{1},E'_{2})$, then we have
   $$
      \sigma_1\hat{\sigma}_2=\delta_1\hat{\delta_2}.
   $$
\end{lemma}

\begin{proof}
   There are $c_1,c_2\in \C$ with
   $c_i\Lambda_i=\Lambda'_i$. We have then
   $$
      \Hom(E'_{1},E'_{2})=\frac{c_2}{c_1}\Hom(E_{1},E_{2}).
   $$
   and therefore
   $$
      \delta_1\hat{\delta_2}
      =\frac{c_2}{c_1}\sigma_1\widehat{(\frac{c_2}{c_1}\sigma_2)}
      =\frac{a(\Lambda'_1)}{a(\Lambda'_2)}\frac{|c_2|^2}{|c_1|^2}\sigma_1\overline\sigma_2
      =\frac{a(\Lambda_1)}{a(\Lambda_2)}\sigma_1\overline\sigma_2=\sigma_1\hat{\sigma}_2.
   $$
\end{proof}

\begin{proofT}
   After using a suitable isomorphism we can assume $\Lambda_1\subset\Lambda_2$. With the elementary divisor theorem we get
   $$
      \Lambda_1=\Z e_1u+\Z e_2 v\subset \Z u+\Z v =\Lambda_2
   $$
   with $u,v\in\C$ and $e_1,e_2\in\N$ with $e_1\divides e_2$.
   Using again suitable isomorphisms we can
   assume
   (with Lemma~\ref{lemma:isom})
   that
   $$
      \Lambda_1=\Z+\Z t\tau\subset \Z+\Z \tau =\Lambda_2
   $$
   with $t\in\N$ and $\tau \in \C\setminus \R$.

   As above, we have $\Hom(E_{1},E_{2})=\Z\sigma_1+\Z\sigma_2$
   for some $\sigma_1,\sigma_2\in \C$.
   Because of $\Lambda_1\subset\Lambda_2$,
   the identity is an element of
   $\Hom(E_{1},E_{2})$.
   Furthermore, $\Z\cdot 1$ is a primitive sublattice of
   $\Hom(E_{1},E_{2})$, since
   $\Hom(E_{1},E_{2})$ is a sublattice of $\Lambda_2$,
   and the latter
   contains no
   elements of $\Q\setminus\Z$.
   This shows that
   there is
   a basis of $\Hom(E_{1},E_{2})$ of the form $1,\sigma$.
   We fix such a basis for the remainder of the proof.

\medskip
   For any homomorphism $\lambda:E_1\to E_2$ we have
   $$
      \widehat{\lambda}=\frac{a(\Lambda_1)}{a(\Lambda_2)}\cdot\overline\lambda
      =\frac{|t\cdot \textnormal{Im}(\tau)|}{| \textnormal{Im}(\tau)|}\cdot\overline\lambda
      =t\cdot\overline\lambda.
   $$
   Because of $1\cdot\widehat{\lambda}\in \textnormal{End}(E_2)$ and $\widehat{\lambda}\cdot 1\in \textnormal{End}(E_1)$ we get
   $$
      \lambda \in \frac{1}{t}\cdot\left( \Z+\Z\cdot \textnormal{LCM}(f_1,f_2)\cdot\overline\omega\right).
   $$
   So there is a $k\in \Z$ with
   $$
      \textnormal{Im}(\sigma)=k\cdot\frac{\textnormal{LCM}(f_1,f_2)\cdot\textnormal{Im}(\omega)}{t}.
      \eqno(*)
   $$

\medskip
   The number $\kappa:=\textnormal{LCM}(f_1,f_2)\cdot \omega$ defines an endomorphism of $E_1$,
   so $\kappa\cdot 1\in\Lambda_1$, and hence
   there are integers $x,y$ with $\kappa=x+y\cdot t\tau$.
   Also, $\kappa$ defines an endomorphism of $E_2$, and hence $\kappa\tau\in\Lambda_2$.
   Now the number
   $$
      y\tau=\frac{-x+\kappa}{t}
   $$
   defines a homomorphism $E_1\to E_2$, because
   $y\tau\cdot 1\in\Lambda_2$ and
   $$
      y\tau\cdot t\tau=
      \frac{-x+\kappa}{t}\cdot t\tau=-x\tau+\kappa\tau\in\Lambda_2
      \,.
   $$
   So we found that
   $\frac{-x+\kappa}{t}\in \Z+\Z\sigma$.
   Comparing now the
   imaginary part
   with $(*)$ we conclude that
   $k=\pm 1$. Therefore we obtain, using $(*)$,
   $$
      |\textnormal{Im}(\sigma_1\cdot\hat{\sigma_2})|
      =|\textnormal{Im}(1\cdot\hat{\sigma})|
      =\left| \textnormal{Im}\left(t\cdot \overline{\sigma}\right)\right|
      =\textnormal{LCM}(f_1,f_2)\cdot |\textnormal{Im}(\omega)|
      \,,
   $$
   and this proves the theorem.
\end{proofT}

\begin{remark}\rm
   A self-product $E\times E$ of an elliptic curve with complex multiplication
   is a special case of Type C:
   In that situation
   we can take
   $\sigma_1=1$ and $\sigma_2=f\omega$ as generators of
   $\Hom(E_1,E_2)$, and get the discriminant
   \be
      \Delta(X)=\begin{bycases}
                -4 f^2|d| & \mbox{ if } d\equiv 2,3   \tmod 4 \\
                 - f^2|d|     & \mbox{ if } d\equiv 1 \tmod 4
             \end{bycases}.
   \ee
\end{remark}



\footnotesize
   \bigskip
   Thomas Bauer,
   Fachbereich Mathematik und Informatik,
   Philipps-Universit\"at Marburg,
   Hans-Meerwein-Stra\ss e,
   D-35032 Marburg, Germany.

   \nopagebreak
   \textit{E-mail address:} \texttt{tbauer@mathematik.uni-marburg.de}

   \bigskip
   Carsten Borntr\"ager
   Fachbereich Mathematik und Informatik,
   Philipps-Universit\"at Marburg,
   Hans-Meerwein-Stra\ss e,
   D-35032 Marburg, Germany.

   \nopagebreak
   \textit{E-mail address:} \texttt{borntraegerc@mathematik.uni-marburg.de}


\end{document}